\documentclass[12pt]{article}

\usepackage{geometry} 
\usepackage{graphicx}
\usepackage{color}
\usepackage{amsfonts}
\usepackage{amsmath}
\usepackage{booktabs} 
\usepackage{array} 
\usepackage{paralist} 
\usepackage{verbatim} 
\usepackage{subfig} 

\def\P{{\cal P}}

\def\b.#1{{\bf #1}}  

\def\cuadrito{ ${\vcenter{\vbox{\hrule height.4pt
                                \hbox{\vrule width.4pt height3pt \kern3pt
                                      \vrule width.4pt}
                                 \hrule height.4pt}}}$ }

\def\cua{ ${\vcenter{\vbox{\hrule height .4pt
                                \hbox{\vrule width .4 pt height7pt \kern7pt
                                      \vrule width .4pt}
                                 \hrule height .4pt}}}$ }

 \newtheorem{theorem}{Theorem}

 \newtheorem{claim}{Claim}

 \newtheorem{corollary}[theorem]{Corollary}

 \newtheorem{lemma}{Lemma}

 \newtheorem{proposition}[theorem]{Proposition}

 \newenvironment{proof}[1][Proof]{\textbf{#1.} }{\ \rule{0.5em}{0.5em}\vskip 12pt}
\usepackage{color}

\newif\ifpdf
\ifx\pdfoutput\undefined
\pdffalse 
\else
\pdfoutput=1 
\pdftrue
\fi

\title{Decomposition of  balanced multipartite tournaments into  strongly connected tournaments\footnote{Research supported  by  PAPIIT-M\'exico under project IN104915.}}
\author{Ana Paulina Figueroa\footnote{Departamento de Matem\'aticas ITAM, M\'exico. email: apaulinafg@gmail.com}\cr Juan Jos\'e Montellano-Ballesteros\footnote{Instituto de Matem\'aticas, UNAM, M\'exico, email: juancho@im.unam.mx}\cr Mika Olsen\footnote{Departamento de Matem\'aticas Aplicadas y Sistemas, UAM-C, M\'exico. email:olsen@correo.cua.uam.mx}}

\date{}{}{}
\begin{document}

\maketitle





\begin{abstract} Decomposing a digraph into  subdigraphs with a fixed structure or property is a classical problem in graph theory and a useful tool in a number of applications of networks and communication. A digraph is strongly connected if  it contains a directed path from each vertex to all others.  In this paper we consider multipartite tournaments,  and we study the existence of a partition of a multipartite tournament  with $c$ partite sets into strongly connected $c$-tournaments. This is a continuation of the study started in 1999 by Volkmann of the existence of strongly connected subtournaments in multipartite tournaments.
\end{abstract}

\noindent{\bf Keyword: } Oriented graphs, Multipartite tournaments, Decomposition,  Strong connected digraphs


\section{Introduction and definitions}
\label{sec1} Decomposing a digraph into  subdigraphs with a fixed structure or property is a classical problem in graph theory and a useful tool in a number of applications of networks and communication. For instance, finding a decomposition in strongly connected components has been used in  compiler analysis, data mining, scientific computing, social networks and other areas. 
 In this paper we consider multipartite tournaments,  and we study the existence of a partition of the set of vertices of  a multipartite tournament  with $c$ partite sets, into strongly connected tournaments of order $c$. Observe that  every partite set of the multipartite tournament has  exactly one  vertex in each strongly connected tournament of the partition.   We can illustrate our result with the following situation:  if  all the vertices of any partite set has the same information, and any pair of vertices of different partite sets has different information, then the total information spread  among all the vertices of the digraph can be distributed effectively using the partition into strongly connected tournaments  since each strongly connected tournament possess one vertex of each partite set.

Let $c$- be a non-negative integer, a  {\it $c$-partite or multipartite tournament} is a digraph obtained from a complete $c$-partite graph by orienting each edge. 
In 1999 \cite{sub} Volkmann developed the first contributions in the study of the structure of the strongly connected subtournaments in multipartite tournaments. He proved that every almost regular $c$-partite tournament contains a strongly connected subtournament of order $p$ for each $p\in\{3,4, \ldots, c-1 \} $. In the same paper he also proved that if each partite set of an almost regular  $c$-partite tournament has at least $\frac{3c}{2}-6$ vertices, then there exist a strong subtournament of order $c$. In 2008 \cite{Almost-regular}, Volkmann and Winsen proved that every almost regular $c$-partite tournament has a strongly connected subtournament of order $c$ for $c\geq 5$. In 2011 \cite{X} Xu et al. proved that every  vertex of regular  $c$-partite tournament with  $c\ge16$, is contained in a strong subtournament of order  $p$ for every  $p\in \{3,4,\dots,c\}$. Finally, in 2016 \cite{amo},  we proved that for every (not necessarily strongly connected) balanced $c$-partite tournament of order $n\ge6$, if  the global irregularity of $T$ is at most $\frac{c}{\sqrt{3c+26}}$, then $T$ contains a strongly connected tournament of order $c$. 


Let $T$ be a $c$-partite tournament of order $n$ with partite sets $\{V_i\}_{i=1}^c$. We call $T$ \textit{balanced}, if all partite sets contain the same number of vertices and we denote by $G_{r,c}$ a balanced $c$-partite tournament satisfying that $|V_i|=r$ for every $1\le i\le c$. Throughout this paper $|V_i|=r$ for each $i\in [c]$. 
{As a {\it partition of $G_{r,c}$ in maximal tournaments} we will understand a spanning subdigraph of $G_{r,c}$  which is a set of $r$ pairwise vertex-disjoint tournaments of order $c$}.

Our main result  gives sufficient conditions in terms of the minimum degree, the number of partite and its order to guarantee that a $r$-balanced $c$-partite tournament has a partition 
{in maximal tournaments such that each of its $r$ tournaments is strongly connected. Such a partition will be called a \textit{strong partition}.}

We will follow almost all the definitions and notation of \cite{B-J-G}. The maximal independent sets of $T$ are called the partite sets of $T$. If $T$ is a $c$-partite tournament, $\Delta(T)=\max\{d^j(x):x\in V(T), j\in\{+,-\} \}$  and $\delta(T)=\min\{d^j(x):x\in V(T), j\in\{+,-\} \}$. Notice that $i_g(T)=\Delta(T)-\delta(T)$. Let $x\in V(T)$ and  $i\in [c]$, the out-neighborhood of $x$ in $V_i$ is $N^+_i(x) =V_i\cap N^+(x)$; the in-neighborhood of $x$ in $V_i$ is $N^-_i(x) =V_i\cap N^-(x)$;   $d^+_i(x) =|N^+_i(x) |$ and   $d^-_i(x) = |N^-_i(x)|$. If $x\in V(T)$,  $\Delta_V^+(x)= \max \{d_i^+(x): i\in [c] \}$; $\delta_V^+(x)= \min \{d_i^+(x): i\in [c] \}$);   $\Delta_V^-(x)= \max \{d_i^-(x): i\in [c] \}$ and $\delta_V^-(x)= \min \{d_i^-(x): i\in [c] \}$.
If $T$ is a $c$-partite tournament, the \textit{maximum out-degree of $T$ with respect to the parts} is $\Delta_V^+(T) = \max \{d_i^+(x) : i \in [c] \hbox{ \& } x\in V(T)\}$ and the \textit{minimum out-degree of $T$ with respect to the parts} is $\delta_V^+(T) = \min \{ d_i^+(x) : i \in [c] \hbox{ \& } x\in V(T)\}$. Analogously, we define   $\Delta_V^-(T) = \max \{d_i^-(x) : i \in [c] \hbox{ \& } x\in V(T)\}$ and $\delta_V^+(T) = \min \{ d_i^-(x) : i \in [c] \hbox{ \& } x\in V(T)\}$. We will simply write $\Delta^+_V$, for example, instead of $\Delta^+_V(T)$ whenever it is {clear} in which $c$-partite tournament $T$ we are working on. 

If $G_{c,r}$ is a balanced $c$-partite tournament, we define a new measure of irregularity called the irregularity restricted to the parts as $\mu (G_{r,c})= \max\{\Delta^+_V-\delta_V^+, \Delta^-_V-\delta^-_V \}$. Our main result has $\mu(G_{r,c})$ as a parameter.

\section{Main Results} 

In this section we used Lemmas $1-4$ in order to proof our Main Result. The  prove of these  lemmas can be found in Section  \ref{pruebas lemas}.

\begin{lemma}\label{particion}  The number of partitions of $G_{r,c}$  {in maximal tournaments} is $(r!)^{c-1}$.\end{lemma}

Let  $x\in V_c$ and  let $\mathcal{H}_k(x)$ be the set of vectors  $(h_1,h_2, \ldots h_{c-1})\in \{0,1\}^{c-1}$ such that  $h_i=1$ if $d_i^+(x)=r$, $h_i=0$ if $d_i^+(x)=0$ and $\sum_{i=1}^{c-1} h_i=k$.

\begin{lemma}\label{L M(gi)}
Let $G_{r,c}$ be a balanced $c$-partite tournament and let $x\in V_c$. The number of maximal {tournaments of $G_{r,c}$} for which $x$ has out-degree $k$ is equal to 
$$\sum_{h\in \mathcal{H}_k(x)}\prod_{i=1}^{c-1}d_i^+(x)^{h_i}d_i^-(x)^{1-h_i}.$$   
\end{lemma}

For each $x\in V(T)$ let $T^+(x)$  (resp.  $T^-(x)$ ) be the number of maximal {tournaments}  of $G_{r,c}$ for which $x$ has out-degree (resp. in-degree)  at most $\lceil\frac{c-2}{4}\rceil$.  The following Lemma provides an upper bound for $T^+(x)$. An analogous result for $T^-(x)$ can be obtained using similar arguments.

\begin{lemma}\label{promedio}  Let $G_{r,c}$  be a balanced $c$-partite tournament such that  $\delta(G_{r,c}) \geq \lfloor\frac{c-2}{4}\rfloor (r + \mu) + max \{\delta^+_V, \delta^-_V\}$.  
Then, for every $x\in V(G_{r,c})$, $$T^+(x) \leq \sum\limits_{k=0}^{\left\lfloor \frac{c-2}{4} \right\rfloor} {{c-1}\choose {k}} \left(\frac{d^+(x)}{d^-(x)}\right)^k \left(\frac{d^-(x)}{c-1}\right)^{c-1}.$$

Moreover, if  $\lceil\frac{c-2}{4}\rceil  < |\{i : d_i^+(x) = r\}| $,  $T^+(x) = 0$.  
\end{lemma}

\begin{lemma}\label{suma} Let $G_{r,c}$  be a balanced $c$-partite tournament, with $c\geq 10$, such that  $\delta(G_{r,c}) \geq r(c-1)\left(\frac{c+6}{4(c+1)}\right)$. Then, for every $x\in V_c$, 
 \[
 \sum\limits_{k=0}^{ \lfloor\frac{c-2}{4}\rfloor}{{c-1}\choose {k}} \left(\frac{d^+(x)}{d^-(x)}\right)^k  < 
\begin{cases}
{{c-1}\choose {\lfloor\frac{c-2}{4}\rfloor}}\left(\frac{3c-2}{2c-4}\right) \left(\dfrac{d^+(x)}{d^-(x)}\right)^{ \lfloor\frac{c-2}{4}\rfloor} & \hbox{ if }  d^+(x)\geq d^-(x); \\
{{c-1}\choose {\lfloor\frac{c-2}{4}\rfloor}}\left(\frac{3c-2}{2c-4}\right) \left(\dfrac{d^-(x)}{d^+(x)}\right)^{ \lfloor\frac{c-2}{4}\rfloor} & \hbox{ if }  d^+(x)<d^-(x) .
\end{cases}
\] \end{lemma}

Let $\P$ be the set of all the partitions  of $G_{r,c}$ in maximal tournaments. For each partition  in maximal tournaments, $P$ of $G_{r,c}$ let $\omega (P)$ be the number of vertices $x\in V(G_{r,c})$ such that $\delta_P(x) \leq  \lfloor\frac{c-2}{4}\rfloor$ and let $\omega (G_{r,c}) = \sum\limits_{P \in \P} \omega(P)$.

\begin{theorem}\label{main} Let $G_{r, c}$ be a balanced  $c$-partite tournament, with $c\geq 10$,  $\alpha=\frac{2\Delta (G_{r,c})}{r(c-1)}$ and $\beta=\frac{2\delta (G_{r,c})}{r(c-1)}$.  
 $G_{r, c}$ has a strong partition if

$$\frac{\omega(G_{r,c})}{((r-1)!)^{c-1} rc} \geq {{c-1}\choose {\lfloor\frac{c-2}{4}\rfloor}} \left(\frac{3c-2}{2c-4}\right)  \left(\frac{\alpha}{\beta}\right)^{ \lfloor\frac{c-2}{4}\rfloor}\left(\frac{r}{2}\right)^{c-1}\left(\alpha^{c-1} + \beta^{c-1}\right)$$

 and  $$\delta (G_{r,c}) \geq  \max\left\{r(c-1)\left(\frac{c+6}{4(c+1)}\right),\left\lfloor\frac{c-2}{4}\right\rfloor (r+\mu) + \max\{\delta_V^+,\delta_V^-\}\right \}.$$
\end{theorem}

\begin{proof}

Let $G_{r, c}$ be such that $\delta (G_{r,c}) \geq \max\{ r(c-1)\frac{c+6}{4(c+1)}, \lfloor\frac{c-2}{4}\rfloor(r+\mu) + \max\{\delta^+_V,\delta^-_V \}\}$ and suppose there is no strong partition.   
For each $x\in V(G_{r,c})$, let $F^+(x)$ (resp. $F^-(x)$ )  be the number of partitions  $P$ of $G_{r,c}$ for which  $d_P^+(x) \leq \lfloor\frac{c-2}{4}\rfloor$ (resp. $d_P^-(x) \leq \lfloor\frac{c-2}{4}\rfloor$).  
If follows that   $\sum\limits_{x\in V(G_{r,c}) } (F^+(x)+ F^-(x))  = \omega (G_{r,c})$ and, by an average argument, we can notice that there exists $x_0\in V(G_{r,c})$ such that 
\begin{equation}\label{m1} F^+(x_0) + F^-(x_0) \geq\frac{\omega (G_{r,c})}{rc}.
\end{equation}
Notice that each maximal tournament is a member of  $((r-1)!)^{c-1}$ partitions $P$ of $G_{r,c}$ in maximal tournaments (using the same argument as in the proof of Lemma 1). Then,  $F^+(x_0) = ((r-1)!)^{c-1} T^+(x_0) $ and $F^-(x_0) = ((r-1)!)^{c-1} T^-(x_0)$. Thus, by (\ref{m1}), 
\begin{equation}\label{m2} 
T^+(x_0) + T^-(x_0) \geq  \frac{\omega(G_{r,c})}{((r-1)!)^{c-1}rc }.
\end{equation}

Assume, w.l.o.g, $d^+(x_0)\geq d^-(x_0)$. Since $\delta (G_{r,c}) \geq  \lfloor\frac{c-2}{4}\rfloor(r+\mu) + \max\{\delta^+_V,\delta^-_V \}$, by Lemma \ref{promedio}, 
$$T^+(x_0)\leq \sum\limits_{k=0}^{\left\lfloor \frac{c-2}{4} \right\rfloor} {{c-1}\choose {k}} \left(\frac{d^+(x_0)}{d^-(x_0)}\right)^k \left(\frac{d^-(x_0)}{c-1}\right)^{c-1}.$$ 

By hypothesis, $\delta (G_{r,c}) \geq  r(c-1)\frac{c+6}{4(c+1)}$, then by Lemma \ref{suma} 

$$T^+(x_0)<  {{c-1}\choose {\lfloor\frac{c-2}{4}\rfloor}}\left(\frac{3c-2}{2c-4}\right) \left(\dfrac{d^+(x_0)}{d^-(x_0)}\right)^{ \lfloor\frac{c-2}{4}\rfloor} \left(\frac{d^-(x_0)}{c-1}\right)^{c-1},$$
and since $\frac{d^+(x_0)}{d^-(x_0)}\leq \frac{\Delta(G_{r,c})}{\delta(G_{r,c})}= \frac{\alpha}{\beta}$
$$T^+(x_0)< {{c-1}\choose {\lfloor\frac{c-2}{4}\rfloor}}\left(\frac{3c-2}{2c-4}\right) \left(\dfrac{\alpha}{\beta}\right)^{ \lfloor\frac{c-2}{4}\rfloor} \left(\frac{d^-(x_0)}{c-1}\right)^{c-1}.$$

Analogously, since  $d^+(x_0)\geq d^-(x_0)$, by Lemmas \ref{promedio} and \ref{suma}, we can see that 
$$T^-(x_0) <  {{c-1}\choose {\lfloor\frac{c-2}{4}\rfloor}}\left(\frac{3c-2}{2c-4}\right) \left(\dfrac{\alpha}{\beta}\right)^{ \lfloor\frac{c-2}{4}\rfloor} \left(\frac{d^+(x_0)}{c-1}\right)^{c-1}.$$

Thus,  by (\ref{m2}),  
$$\frac{\omega(G_{r,c})}{((r-1)!)^{c-1}rc} <  {{c-1}\choose {\lfloor\frac{c-2}{4}\rfloor}}\left(\frac{3c-2}{2c-4}\right) \left(\dfrac{\alpha}{\beta}\right)^{ \lfloor\frac{c-2}{4}\rfloor} \left(\left(\frac{d^-(x_0)}{c-1}\right)^{c-1}+  \left(\frac{d^+(x_0)}{c-1}\right)^{c-1}\right).$$  

Finally, since $d^+(x_0) + d^-(x_0) = r(c-1) = \Delta(G_{r,c}) + \delta(G_{r,c})$, with $\Delta(G_{r,c}) \geq d^+(x_0) \geq d^-(x_0)\ge\delta(G_{r,c})$, we see that  $$\left(\dfrac{d^-(x_0)}{c-1}\right)^{c-1} + \left(\dfrac{d^+(x_0)}{c-1}\right)^{c-1}\leq  \left(\dfrac{\delta(G_{r,c})}{c-1}\right)^{c-1} + \left(\dfrac{\Delta(G_{r,c})}{c-1}\right)^{c-1},$$  and $ \left(\dfrac{\delta(G_{r,c})}{c-1}\right)^{c-1} + \left(\dfrac{\Delta(G_{r,c})}{c-1}\right)^{c-1} =\left(\frac{\beta r}{2}\right)^{c-1} + \left(\frac{\alpha r}{2}\right)^{c-1} = \left(\frac{r}{2}\right)^{c-1}\left(\alpha^{c-1} + \beta^{c-1}\right). $

Therefore,
$$\frac{\omega(G_{r,c})}{((r-1)!)^{c-1} rc} < {{c-1}\choose {\lfloor\frac{c-2}{4}\rfloor}}\left(\frac{3c-2}{2c-4}\right) \left(\frac{\alpha}{\beta}\right)^{ \lfloor\frac{c-2}{4}\rfloor}\left(\frac{r}{2}\right)^{c-1}\left(\alpha^{c-1} + \beta^{c-1}\right)$$
and the result follows. 
 \end{proof}

\begin{corollary}\label{unmalo}
 Let $G_{r, c}$ be a balanced  $c$-partite tournament, with $c\geq 10$,  $\alpha=\frac{2\Delta (G_{r,c})}{r(c-1)}$ and $\beta=\frac{2\delta (G_{r,c})}{r(c-1)}$.  Then, $G_{r, c}$ has a strong partition if

$$\frac{1}{r} \geq\sqrt{c}\left(\frac{9}{7}\right)\left(\frac{2}{3^\frac{3}{4}}\right)^{c-1}  \left(\frac{\alpha}{\beta}\right)^{ \lfloor\frac{c-2}{4}\rfloor}\left(\alpha^{c-1} + \beta^{c-1}\right)$$

 and  $$\delta (G_{r,c}) \geq  \max\left\{r(c-1)\left(\frac{c+6}{4(c+1)}\right),\left\lfloor\frac{c-2}{4}\right\rfloor (r+\mu) + \max\{\delta_V^+,\delta_V^-\}\right \}.$$
\end{corollary}

\begin{proof}

 Let $G_{r, c}$ be a balanced  $c$-partite tournament, with $c\geq 10$, and with no strong partition. For each partition $P$ of $G_{r,c}$, $\omega (P) \geq 1$, because  every tournament of order $c$ that is not strongly connected has minimum degree at most  $\lfloor\frac{c-2}{4}\rfloor$ and since $P$ is not a strong partition, at least one of its tournaments is not strongly connected. Thus, by Lemma \ref{particion},  $\omega(G_{r,c}) \geq (r!)^{c-1}$. By  the Main Theorem,

$$\frac{(r!)^{c-1}}{((r-1)!)^{c-1} rc} \geq {{c-1}\choose {\lfloor\frac{c-2}{4}\rfloor}} \left(\frac{3c-2}{2c-4}\right)  \left(\frac{\alpha}{\beta}\right)^{ \lfloor\frac{c-2}{4}\rfloor}\left(\frac{r}{2}\right)^{c-1}\left(\alpha^{c-1} + \beta^{c-1}\right).$$ Simplifying this inequality we obtain that
\begin{equation}\label{sinomega}
\frac{1}{rc} \geq {{c-1}\choose {\lfloor\frac{c-2}{4}\rfloor}} \left(\frac{3c-2}{2c-4}\right)  \left(\frac{\alpha}{\beta}\right)^{ \lfloor\frac{c-2}{4}\rfloor}\left(\frac{1}{2}\right)^{c-1}\left(\alpha^{c-1} + \beta^{c-1}\right)
\end{equation}
We will prove by induction that  
\begin{equation}\label{bound}
 p(c): {{c-1}\choose {\lfloor\frac{c-2}{4}\rfloor}} \leq \left(\frac{9}{7\sqrt{c}}\right)\left(\frac{4}{3^\frac{3}{4}}\right)^{c-1} \left(\frac{2c-4}{3c-2}\right) 
\end{equation}
Base cases holds:

\begin{tabular}{c|c|l}
$c$ & ${{c-1}\choose {\lfloor\frac{c-2}{4}\rfloor}}$  & $\leq$ $\left(\frac{9}{7\sqrt{c}}\right)\left(\frac{4}{3^\frac{3}{4}}\right)^{c-1} \left(\frac{2c-4}{3c-2}\right)$\\
\hline
10 & 36 & 36.65\\
11 & 45 & 62.3\\
12 & 55 & 105.05 \\
13& 66 & 180.72
\end{tabular}

For the inductive step we will prove that $p(c)$ implies $p(c+4)$. 

$$\begin{array}{lcl} 
{{c+3}\choose{\lfloor\frac{c-2}{4}\rfloor}+1}&=&\frac{c(c+1)(c+2)(c+3)}{(c+2-\lfloor\frac{c-2}{4}\rfloor)(c+1-\lfloor\frac{c-2}{4}\rfloor)(c-\lfloor\frac{c-2}{4}\rfloor)(\lfloor\frac{c-2}{4}\rfloor+1)} {{c-1}\choose {\lfloor\frac{c-2}{4}\rfloor}}\\
& \leq & \frac{4^4}{3^3}\frac{3c(3c+3)}{(3c+13)(3c+5)}{{c-1}\choose {\lfloor\frac{c-2}{4}\rfloor}}\\
&\leq & \frac{4^4}{3^3}\frac{3c(3c+3)}{(3c+13)(3c+5)}\left(\frac{9}{7\sqrt{c}}\right)\left(\frac{4}{3^\frac{3}{4}}\right)^{c-1} \left(\frac{2c-4}{3c-2}\right) \\
&=&\left(\frac{4}{3^\frac{3}{4}}\right)^{c+3}\left(\frac{9}{7\sqrt{c}}\right)\frac{3c(3c+3)(2c-4)}{(3c+13)(3c+5)(3c-2)}
\end{array}$$

To complete the induction it is enough to prove that, for $c\geq 10$, $$\left(\frac{1}{\sqrt{c}}\right)\frac{3c(3c+3)(2c-4)}{(3c+13)(3c+5)(3c-2)}\leq \left(\frac{1}{\sqrt{c+4}}\right) \left(\frac{2c+4}{3c+10}\right).  $$ 
This is a consequence of the fact that the real function $$f(c)= \left(\frac{1}{\sqrt{c}}\right)\frac{3c(3c+3)(2c-4)}{(3c+13)(3c+5)(3c-2)}-\left(\frac{1}{\sqrt{c+4}}\right) \left(\frac{2c+4}{3c+10}\right)$$ has no roots for positive $c\geq 10$  and $f(10)<0$, which can be proved by computer.

Therefore, $p(c+4)$ holds. By \ref{sinomega} and \ref{bound}, the result follows.\end{proof}

Notice that if $G_{r,c}$ is regular, $\alpha = \beta = 1$ and since $\delta (G_{r,c}) = \frac{r(c-1)}{2},$ we have the following corollary.

\begin{corollary}\label{regular}
 Let $G_{r, c}$ be a regular balanced  $c$-partite tournament, with $c\geq 10$.
 $G_{r, c}$ has a strong partition if

$$\frac{1}{r}\geq\sqrt{c}\left(\frac{18}{7}\right)\left(\frac{2}{3^\frac{3}{4}}\right)^{c-1} . $$

\end{corollary}

Let $C(r)$ be the number of partite sets such that if $c\geq C(r)$, every  balanced  multipartite tournament $G_{c,r}$ satisfying the minimum degree condition of the Main Theorem has a  strong partition. By the Main Theorem, $C(r)$ is bounded for each non negative integer $r$. In the following table we illustrate the upper bounds of $C(r)$ for regular balanced  multipartite tournament given by Corollary \ref{regular}.

\begin{center}
\begin{tabular}{c|ccccc}
$r$ &2 & 3&5&10&100 \\
\hline
$C(r) \leq $& 26&30&35&40&60

\end{tabular}
\end{center}

\section{Proofs of the lemmas}\label{pruebas lemas}

\noindent {\bf Proof of Lemma \ref{particion}. }  

Let $G_{r,c}$ be a balanced $c$- partite tournament and let $V_1, \ldots V_c$ its partite sets. Let $V_1 = \{x_1, \dots, x_r\}$ and,  given a partition of $G_{r,c}$ in maximal tournaments, let $\tau_1, \tau_2, \ldots \tau_r$ be its set of tournaments.  W.l.o.g, assume that for every partition of $G_{r,c}$,  the vertex $x_i$  of $V_1$ is a vertex of  $\tau_i$. Then, every partition corresponds to  a permutation of the vertices in each $V_j$ {(for $2\leq j \leq c$)} choosing the $i$-th member of the permutation of $V_j$ to be in $\tau_i$. Since there are $r!$ permutations of each $V_j$, $j\in \{2, 3, \ldots, c \}$, the result follows. 
\hfill${\ \rule{0.5em}{0.5em}}$\\

\noindent\textbf{Proof of Lemma \ref{L M(gi)}}	
	
Let $x\in V(G_{r,c})$. A maximal tournament containing the vertex $x$ with out-degree $k$ can be constructed choosing a vertex for each part $V_i$ for $1\le i\le c-1$ in the following way. Given $\textbf{h}=(h_1,h_2,\dots,h_{c-1})\in \mathcal{H}_k(x)$, we choose an out neighbor of $x$ from $V_i$ if and only if $h_i=1$. The number of maximal tournaments constructed in this way for a fixed $\textbf{h}$ is $\prod_{i=1}^{c-1}d_i^+(x)^{h_i}d_i^-(x)^{1-h_i} $. 
Therefore, for each $\textbf{h}\in \mathcal{H}_k(x)$, there are $d_i^+(x)^{h_i}d_i^-(x)^{1-h_i}$ ways to construct such a maximal {tournaments} and the result follows. \hfill${\ \rule{0.5em}{0.5em}}$\\

\noindent {\bf Proof of Lemma \ref{promedio}. } 
	
Let $x \in V(G_{r,c})$. Assume w.l.o.g that  $x\in V_c$, let $M(d_1^+(x),d_2^+(x),\dots,d_{c-1}^+(x);k)$ be the function that calculates the number of maximal {tournaments} for which $x$ has out-degree $k$. Then by Lemma \ref{L M(gi)},  $$M(d_1^+(x),d_2^+(x),\dots,d_{c-1}^+(x);k)=
\sum_{h\in \mathcal{H}_k(x)}\prod_{i=1}^{c-1}
d_i^+(x)^{h_i}d_i^-(x)^{1-h_i}$$
and 
$$T^+(x) = \sum\limits_{k=0}^{\lfloor\frac{c-2}{4}\rfloor} M(d_1^+(x), d_2^+(x) \ldots d_{c-1}^+(x);k).$$

To give an upper bound of $T^+(x)$ we need to extend the definition of  $M(d_1^+(x),d_2^+ (x)\dots , d_{c-1}^+(x); k)$  to the real numbers, as follows:  

Let $g_1,g_2, \ldots,g_{s}$ be real numbers such that $0 \le g_i \leq r$, 
we define
 $$M(g_1, \dots , g_s; k)=\sum_{h\in \mathcal{H}_{k,s}}\prod_{i=1}^{s} g_i^{h_i}(r-g_i)^{1-h_i}, $$
 
where $\mathcal{H}_{k,s}$ is the set of $s$-vectors  $(h_1,h_2, \ldots,h_s)\in \{0,1\}^s$ such that  if $g_i=r$, $h_i=1$;  if $g_i=0$,  $h_i=0$, and  $\sum_{i=1}^sh_i=k$. For the sake of  readability, in what follows   $M(g_1, \dots , g_s; k)$ can be denoted as $M(g_{[s]}; k).$ In order to prove the Lemma \ref{promedio}, we will prove the following general version:
 
\begin{proposition}\label{general} Let $g_1, g_2, \dots , g_{c-2}, g_{c-1},q$ be real numbers such that  $0\leq g_i \leq r$. Let $p_r=|\{i\in \{1,\ldots,c-3\}: g_i=r  \}|$,  $p_0=|\{i\in \{1,\ldots,c-3\}: g_i=0  \}|$ and $t=c-3-p_r-p_0$ and $p_r+1\leq q.$
 Let $\Gamma= max\{ g_i\}_{i\in[c-1]}$ and $\gamma = min \{g_i\}_{i\in[c-1]}$. 
If $\sum\limits_{i\in [c-1]} g_i \geq q(r+\Gamma-\gamma) + \gamma $, then $$\sum\limits_{k=0}^{q} M(g_{[c-1]}; k)  \leq \sum\limits_{k=0}^{q} {{c-1}\choose {k}} (\epsilon)^k (r-\epsilon)^{c-1-k}$$ for  $\epsilon =  \sum\limits_{i\in [c-1]} \frac{g_i}{c-1}$.\end{proposition}

Assume that $g_1,g_2, \ldots, g_{c-1}$ are ordered in the following way: 
\begin{description}
	\item -   $g_{c-1}= \Gamma$, $g_{c-2}= \gamma$,
	\item - for every  $i \in [t]$, $0< g_i < r$, 
	\item - for every  $t+1 \leq i \leq t+p_r$, $g_i = r$, 
	\item - for every $t+p_r + 1\leq i \leq  t+p_r + p_0 = c-3$, $g_i = 0$.
\end{description}

\begin{claim}\label{afirmacion}
	For every $I\subseteq [t]$ with $|I| = t-(q-p_r) +1$,  $\sum\limits_{i\in I} \frac{g_i}{r-g_i}\geq (q-p_r)$. \end{claim}

$$\centering \begin{array}{lcl}
\sum\limits_{i\in [c-1]} g_i  & = &\sum\limits_{i\in I} g_i +  \sum\limits_{i\in [t]\setminus I} g_i +  \sum\limits_{i= t+1}^{t+p_r} g_i +  \sum\limits_{i = t+ p_r +1}^{c-3} g_i + (g_{c-2} + g_{c-1})\\
&=& \sum\limits_{i\in I} g_i +  \sum\limits_{i\in [t]\setminus I }g_i +  rp_r + (g_{c-2} + g_{c-1}),
\end{array} $$ then, 
$$\centering \begin{array}{lcl}
\sum\limits_{i\in I} g_i &=&  \sum\limits_{i\in [c-1]} g_i  -   \sum\limits_{i\in [t]\setminus I }g_i - rp_r - (g_{c-2} + g_{c-1}).\end{array}$$ 
  
Let  $Q\geq g_i $ for $i\in [t]\setminus I$. Since $\sum\limits_{i\in [c-1]} g_i \geq q(\Gamma+ r-\gamma) + \gamma$, 

$$\begin{array}{lcl}
 \sum\limits_{i\in I} g_i  &\geq &q(\Gamma+ r-\gamma) + \gamma -   Q(q-p_r-1) - rp_r - \Gamma -  \gamma \\ 
 & = &  q(r-\gamma) + (q-1)(\Gamma -Q) + p_r (Q -r).
\end{array}$$ 
If $\Gamma < r$, then  $p_r = 0$, $Q\leq \Gamma$ and therefore, $\sum\limits_{i\in I} g_i  \geq q(r-\gamma)$. 
 If $\Gamma = r$,  since $Q\leq r-1$ and $|I| = t-(q-p_r) +1$,  $\sum\limits_{i\in I} g_i  \geq q(r-\gamma) + (q-1-p_r)(r-Q) \geq  q(r-\gamma).$ 
In both cases, $\sum\limits_{i\in I} g_i  \geq q(r-\gamma)$ and since for every $i \in [c-1]$, $g_i \geq \gamma$, $$\sum\limits_{i\in I} \frac{g_i}{r-g_i} \geq  \sum\limits_{i\in I} \frac{g_i}{r-\gamma}\geq \frac{q(r-\gamma)}{(r-\gamma)}\geq q$$ and the claim follows.
 
 \begin{claim}\label{supremo}
For every $I\subseteq [t]$  such that  $|I| = t-(q-p_r) +1$,  we have that
 $$\sum\limits_{k=0}^{q} M(g_{[c-1]}; k)  \leq \sum\limits_{k=0}^{q} M(g'_{[c-1]}; k) $$ where 
 $g'_{c-2} = g'_{c-1} =  \frac{g_{c-2} + g_{c-1}}{2}$; and for $i\in [c-3]$,  $g'_i = g_i$; 
  \end{claim}
Observe that for every  $k \geq 0$,
$$M(g_{[c-1]}; k) =\sum_{j=0}^2 M(g_{[c-3]}; k-j)M(g_{c-2}, g_{c-1}; j).$$  
Therefore,  for every  $q\geq p_r+1$, 
$$\begin{array}{lcl}
\sum\limits_{k=0}^{q} M(g_{[c-1]}; k) &=& M(g_{[c-3]}; q)M(g_{c-2}, g_{c-1}; 0)\\
& + &M(g_{[c-3]}; q-1)\left[M(g_{c-2}, g_{c-1}; 0)+M(g_{c-2}, g_{c-1}; 1)\right]  \\
 &+& \sum\limits_{k=0}^{q-2} M(g_{[c-3]}; k)\left[M(g_{c-2}, g_{c-1}; 0)+M(g_{c-2}, g_{c-1}; 1)+M(g_{c-2}, g_{c-1}; 2)\right].
\end{array}$$
Notice that  for every pair $x, y \in \mathbb{R}$, such that $0\leq x, y\leq r$,   
\begin{enumerate}
	\item  $M(x, y; 0)+M(x, y; 1)+M(x, y; 2) = (r-x)(r-y) + x(r-y) + (r-x)y + xy= r^2$;  
	\item  $M(x, y; 0)+M(x, y; 1)= (r-x)(r-y) + x(r-y) + (r-x)y = r^2 - xy$ y
	\item  $M(x, y; 0)= r^2 - r(x+y) + xy$. 
\end{enumerate}

Since $g_{c-2}+ g_{c-1} = g'_{c-2}+ g'_{c-1}$, we have that
$$\begin{array}{lcl}
\sum\limits_{k=0}^{q} M(g'_{[c-1]}; k) - \sum\limits_{k=0}^{q} M(g_{[c-1]}; k)&= &M(g_{[c-3]}; q-1)\left[g_{c-2}g_{c-1}- g'_{c-2}g'_{c-1} \right] \\
&+&  M(g_{[c-3]}; q)\left[ g'_{c-2}g'_{c-1} - g_{c-2}g_{c-1} \right]\\
&=& \left( g'_{c-2}g'_{c-1} - g_{c-2}g_{c-1}\right)\left[M(g_{[c-3]}; q) - M(g_{[c-3]}; q-1) \right].
\end{array}  $$

Since  $g'_{c-2}g'_{c-1} \geq  g_{c-2}g_{c-1} $, $\label{cond}\sum\limits_{k=0}^{q} M(g_{[p+2]}; k)\leq \sum\limits_{k=0}^{q} M(g'_{[p+2]}; k)$   if and only if 
\begin{equation}\label{cond1}
M(g_{[c-3]}; q-1)\leq M(g_{[c-3]}; q).
\end{equation}

Let denote by $\mathcal{H}^{c-3}_{q-1} \subseteq \{0,1 \}^{c-3}$, the set of vectors ${\bf h}=(h_1,h_2,\dots,h_{c-3})$ such that  if $g_i=r$, $h_i=1$;  if $g_i=0$,  $h_i=0$, and $\sum_{i=1}^{c-3} h_i=q-1$ .

For each ${\bf h}=(h_1,\ldots. h_{c-3})\in \mathcal{H}^{c-3}_{q-1} $  let 
\begin{equation}\label{defF}
F({\bf h}) = \{(h'_1,\dots, h'_{c-3})\in \mathcal{H}^{c-3}_{q} :  h_i \leq  h'_i  \hbox{ for } i \in [c-3]\} \end{equation}
\noindent {\bf Claim 2.1}  For each ${\bf h}=(h_1,\ldots. h_{c-3})\in \mathcal{H}^{c-3}_{q-1}$.

$$\sum\limits_{\textbf{h'}\in F({\bf h})}   \prod\limits_{i=1}^{c-3} g_i^{h'_i} (r-g_i)^{1-h'_i} \geq (q-p_r) \prod\limits_{i=1}^{c-3} g_i^{h_i} (r-g_i)^{1-h_i}.$$

Let  $\textbf{h}=(h_1,\dots, h_{c-3})\in \mathcal{H}^{c-3}_{q-1}$ and recall that for every  $t+1 \leq j \leq t+p_r$, $g_j = r$;  and for every  $t+ p_r + 1\leq j \leq  t+p_r + p_0 = p$, $g_j = 0$. Therefore, for every  $t+1 \leq j \leq t+p_r$, $h_j = 1$  and for every  $t+ p_r + 1 \leq j \leq c-3$, $h_j = 0$.  By definition, $\sum\limits_{i = 1}^{c-3} h_i = q-1$, therefore w.o.l.g we may assume that  $j \in[q-1- p_r]$; $h_i = 1$ and that $q-p_r \leq j \leq  t$, $h_i = 0$. Hence,
\begin{equation}\label{unos}
(h'_1,\dots, h'_{c-3})\in F(\textbf{h}) \text{ if and only if }   h'_i = 1 \text{  for } i\in[q-1-p_r], \text{ and } \sum\limits_{j=q-p_r}^{t} h'_j =1.
\end{equation}

Therefore, by \ref{defF} and \ref{unos},

$$ \frac{\sum\limits_{\textbf{h'}\in F(\textbf{h})} \prod\limits_{i=1}^{c-3} g_i^{h'_i} (r-g_i)^{1-h'_i}}{  \prod\limits_{i=1}^{c-3} g_i^{h_i} (r-g_i)^{1-h_i}}= \sum\limits_{\textbf{h'}\in F(\textbf{h})} \frac{\prod\limits_{i=1}^{c-3} g_i^{h'_i} (r-g_i)^{1-h'_i}}{  \prod\limits_{i=1}^{c-3} g_i^{h_i} (r-g_i)^{1-h_i}}=\sum\limits_{j= q-p_r}^{t} \frac{g_j}{r-g_j}.$$

By the  hypothesis,  $\sum\limits_{j= q-p_r}^{t} \frac{g_j}{r-g_j}\geq (q-p_r)$, then 
$$  \frac{\sum\limits_{\textbf{h'}\in F(\textbf{h})} \prod\limits_{i=1}^{c-3} g_i^{h'_i} (r-g_i)^{1-h'_i}}{  \prod\limits_{i=1}^{p} g_i^{h_i} (r-g_i)^{1-h_i}}\geq  \left( q -p_r\right)$$
and the Claim 2.1 follows.

Observe that for every $\textbf{h'}\in \mathcal{H}^{c-3}_{q} $ there are exactly  $q-p_r$  elements $\textbf{h}\in \mathcal{H}^{c-3}_{q-1} $ such that, 
$\textbf{h'} \in F(\textbf{h})$. Therefore,
$$ \sum\limits_{\textbf{h}\in H^{c-3}_{q-1}} \left(\sum\limits_{\textbf{h'}\in F(\textbf{h})} \prod\limits_{i=1}^{c-3} g_i^{h'_i} (r-g_i)^{1-h'_i}\right) = (q-p_r) \sum\limits_{\textbf{h'}\in H^{c-3}_{q}}   \prod\limits_{i=1}^{c-3} g_i^{h'_i} (r-g_i)^{1-h'_i}. $$
On the other hand, by Claim 2.1,
 $$ \sum\limits_{\textbf{h}\in H^{c-3}_{q-1}}\left( \sum\limits_{\textbf{h'}\in F(\textbf{h})} \prod\limits_{i=1}^{c-3} g_i^{h'_i} (r-g_i)^{1-h'_i}  \right)\geq  \sum\limits_{\textbf{h}\in H^{c-3}_{q-1}}   \left( q - p_r \right)\prod\limits_{i=1}^{c-3} g_i^{h_i} (r-g_i)^{1-h_i}$$ implying that
$$\sum\limits_{\textbf{h'}\in H^{c-3}_q}  \prod\limits_{i=1}^{c-3} g_i^{h'_i} (r-g_i)^{1-h'_i}\geq  \sum\limits_{\textbf{h}\in H^{c-3}_{q-1}}  \prod\limits_{i=1}^{c-3} g_i^{h_i} (r-g_i)^{1-h_i}$$ which is equivalent to (\ref{cond1}). Therefore Claim 2 follows.\hfill${\ \rule{0.5em}{0.5em}}$\\

 By Claim \ref{afirmacion} and Claim \ref{supremo}  we can conclude that
 $$ \sum\limits_{k=0}^{q} M(g_{[c-1]}, k) \leq \sum\limits_{k=0}^{q} M(g'_{[c-1]}, k)$$  for 
 $g_i = g'_i$, with $i\in [c-3]$;  and  $g'_{c-2} = g'_{c-1} = \frac{g_{c-2} + g_{c-1}}{2}$.  
 
Let $\Gamma' = max\{ g'_i\}_{i\in [c-1]}$ and $\gamma'=  min \{g'_i\}_{i\in[c-1]}$. Since $g_{c-1} \geq g'_{c-1} = g'_{c-2} \geq g_{c-2}$, observe that  $\Gamma  \geq  \Gamma'$ $\gamma\leq \gamma'$.  Therefore $\sum\limits_{i\in [c-1]} g'_i \geq q(\Gamma+ r-\gamma) + \gamma$, as $\sum\limits_{i\in [c-1]} g'_i = \sum\limits_{i\in [c-1]} g_i \geq q(\Gamma+ r-\gamma) + \gamma $ and  $q\geq 1$.
  
We can iterate this process and find $g''_{1}, \dots , g''_{c-1}$ such that   $max\{ g''_i\}_{i\in[c-1]} = min\{ g''_i\}_{i\in[c-1]}$, then, for every $i,j\in [c-1]$,  $g''_i = g''_j$. Hence, for every  $i\in [c-1]$, $g''_i = \frac{\sum\limits_{i\in [c-1]} g''_i}{c-1} = \alpha$, and for each  $q\geq k \geq 0$, $$ M(g''_{[c-1]}; k) =  \sum\limits_{\textbf{h}\in \mathcal{H}^{c-1}_k}  \prod\limits_{i=1}^{c-1} (g''_i)^{h_i} (r-g''_i)^{1-h_i} =$$
 $$\sum\limits_{\textbf{h}\in \mathcal{H}^{c-1}_k}    \prod\limits_{i=1}^{c-1} \alpha^{h_i} (r-\epsilon)^{1-h_i} =\sum\limits_{\textbf{h}\in \mathcal{H}^{c-1}_k}    \alpha^k (r-\epsilon) ^{c-1-k}= {{c-1}\choose {k}}  \alpha^k (r-\epsilon) ^{c-1-k}$$ and the Proposition \ref{general} follows.

 Let $x\in V(G_{r,c})$ and $q= \lfloor\frac{c-2}{4}\rfloor$. Suppose, w.o.l.g., $x\in V_c$. Let  $\Gamma = \Delta^+_V(x) = d^+_{c-1}(x)$; $\gamma = \delta^+_V(x) = d^+_{c-2}(x)$; 
 $p_r=|\{i\in \{1,\ldots,c-3\}: d^+_i(x)=r  \}|$,  $p_0=|\{i\in \{1,\ldots,c-3\}: d^+_i(x)=0  \}|$ and $t=c-3-p_r-p_0$. Observe that  if $\Gamma < r$ then $p_r = 0$, and if $\Gamma = r$,  the vertex $x$ has  out-degree at least $p_r + 1$  in every maximal tournament. If $   \lfloor\frac{c-2}{4}\rfloor < p_r +1$, $T^+(x) = 0$. Thus, we can suppose that $q \geq p_r+1$.  Since $\mu \geq \Delta^+_V(x) - \delta^+_V(x)  = \Gamma - \gamma$ and $ max \{\delta^+_V, \delta^-_V\} \geq  \delta^+_V(x) = \gamma$,  $\sum\limits_{i\in [c-1]} d^+_i(x) \geq   \delta(G_{r,c}) \geq q(r+\Gamma-\gamma) + \gamma$.  By Proposition \ref{general}, $$T^+(x) = \sum\limits_{k=0}^{\lfloor\frac{c-2}{4}\rfloor} M(d^+(x)_{[c-1]}; k)  \leq \sum\limits_{k=0}^{\lfloor\frac{c-2}{4}\rfloor} {{c-1}\choose {k}} (\epsilon)^k (r-\epsilon)^{c-1-k}$$ for  $\epsilon =  \sum\limits_{i\in [c-1]} \frac{d_i^+(x)}{c-1} = \frac{d^+(x)}{c-1}$ and the lemma follows. 
 \hfill${\ \rule{0.5em}{0.5em}}$\\

\noindent {\bf Proof of Lemma \ref{suma}. }    Let $G_{r,c}$  be a balanced $c$-partite tournament such that  $\delta(G_{r,c}) \geq r(c-1)\frac{c+6}{4(c+1)}$, and let $x\in V(G_{r,c})$. In what follows, let $p = d^+(x)$ and $m=d^-(x)$.  Observe that since $\delta(G_{r,c}) \geq r(c-1)\frac{c+6}{4(c+1)}$ and $r(c-1) = p+m$,  $p \geq (p+m)\frac{c+6}{4(c+1)}$. Multiplying the {previous} inequality by $(c+1)$ and adding $-p\frac{c+6}{4}$ we obtain that $p\frac{3c-2}{4}\geq m\frac{c+6}{4}$. Since $\frac{c+6}{4}\geq \lfloor\frac{c-2}{4}\rfloor+2$ and $\frac{3c-2}{4}\leq c-1- \lfloor\frac{c-2}{4}\rfloor$ it follows that 

\begin{equation}\label{pvsm}
 p\left(c-1-\left\lfloor\frac{c-2}{4}\right\rfloor\right)  \geq  m \left(\left\lfloor\frac{c-2}{4}\right\rfloor+2\right)
\end{equation}

\noindent {\bf Claim.} $\sum\limits_{k=0}^{ \lfloor\frac{c-2}{4}\rfloor}{{c-1}\choose {k}} \left(\frac{p}{m}\right)^k  < {{c-1}\choose {\lfloor\frac{c-2}{4}\rfloor}} \left(\frac{p}{m}\right)^{\lfloor\frac{c-2}{4}\rfloor} \frac{p(c-1-\lfloor\frac{c-2}{4}\rfloor)}{p(c-1-\lfloor\frac{c-2}{4}\rfloor) - m(\lfloor\frac{c-2}{4}\rfloor+1)}$

  For each integer  $q\geq 0$, let $g(q) = \sum\limits_{k=0}^q {{c-1}\choose {k}} \left(\frac{p}{m}\right)^k$. 
Observe that,
 $$ \begin{array}{lcl} g(q+1) = 1 +  \sum\limits_{k=1}^{q+1} {{c-1}\choose {k}} \left(\frac{p}{m}\right)^k 
&=& 1 +  \sum\limits_{k=0}^{q} {{c-1}\choose {k+1}} \left(\frac{p}{m}\right)^{k+1} \\
& = & 1 + \frac{p}{m}  \sum\limits_{k=0}^{q} {{c-1}\choose {k+1}} \left(\frac{p}{m}\right)^{k} \\
& = & 1 + \frac{p}{m}  \sum\limits_{k=0}^{q} {{c-1}\choose {k}}\frac{c-1-k}{k+1} \left(\frac{p}{m}\right)^{k} \\
\end{array}$$
 Notice that for each $k\leq q$, $\frac{c-1-q}{q+1}\leq \frac{c-1-k}{k+1}$ therefore

 $$ \begin{array}{lcl}
g(q+1) &\geq& 1 + \frac{p}{m}  \sum\limits_{k=0}^{q} {{c-1}\choose {k}}\frac{c-1-q}{q+1} \left(\frac{p}{m}\right)^{k} \\ 
&=&  1+ \frac{p}{m}\frac{c-1-q}{q+1}  \sum\limits_{k=0}^{q} {{c-1}\choose {k}} \left(\frac{p}{m}\right)^{k}  \\
& >& \frac{p}{m}\frac{c-1-q}{q+1}g(q).
 \end{array}$$

  On the other hand, $g(q+1)= g(q) + {{c-1}\choose {q+1}} \left(\frac{p}{m}\right)^{q+1}$, and therefore  
$$g(q) + {{c-1}\choose {q+1}} \left(\frac{p}{m}\right)^{q+1}> \frac{\frac{p}{m}(c-1-q)}{q+1} g(q).$$ 

 For $\lfloor\frac{c-2}{4}\rfloor = q$,  since  ${{c-1}\choose { q+1}} = {{c-1}\choose { q} }\frac{c-1-q}{q+1}$, multiplying the inequality by $m(q+1)$, adding $-m(q+1)g(q)$ and dividing by $p(c-1-q)-m(q+1)$ (which by (\ref{pvsm}) is positive ), we obtain  $${{c-1}\choose {q}} \left(\frac{p}{m}\right)^{q} \frac{p(c-1-q)}{p(c-1-q) - m(q+1)}> g(q)$$ and from here the claim follows.

By  the Claim, it only remains to prove that

 \[
\left(\frac{p}{m}\right)^{\lfloor\frac{c-2}{4}\rfloor} \frac{p(c-1-\lfloor\frac{c-2}{4}\rfloor)}{p(c-1-\lfloor\frac{c-2}{4}\rfloor) - m(\lfloor\frac{c-2}{4}\rfloor+1)}  < 
\begin{cases}
\left(\frac{3c-2}{2c-4}\right) \left(\dfrac{p}{m}\right)^{ \lfloor\frac{c-2}{4}\rfloor} & \hbox{ if }  p\geq m; \\
\left(\frac{3c-2}{2c-4}\right) \left(\dfrac{m}{p}\right)^{ \lfloor\frac{c-2}{4}\rfloor} &  \hbox{ if } p<m.
\end{cases}
\] 

If $p\geq m$ then,

$$\begin{array}{lcl}
p\left(c-1-\left\lfloor\frac{c-2}{4}\right\rfloor\right) - m\left(\left\lfloor\frac{c-2}{4}\right\rfloor+1\right) &\geq& p\left(c-1-\left\lfloor\frac{c-2}{4}\right\rfloor\right) - p\left(\left\lfloor\frac{c-2}{4}\right\rfloor+1\right)\\
& =&  p\left(c-2-2\left\lfloor\frac{c-2}{4}\right\rfloor\right)
\end{array}$$
 thus, $$\frac{p(c-1-\lfloor\frac{c-2}{4}\rfloor)}{p(c-1-\lfloor\frac{c-2}{4}\rfloor) - m(\lfloor\frac{c-2}{4}\rfloor+1)} \leq \frac{c-1-\lfloor\frac{c-2}{4}\rfloor}{c-2-2\lfloor\frac{c-2}{4}\rfloor}\leq  \frac{3c-2}{2c-4}.$$

For the case when $m>p$, let us suppose on the contrary, that  
$$\left(\frac{p}{m}\right)^{\lfloor\frac{c-2}{4}\rfloor} \frac{p(c-1-\lfloor\frac{c-2}{4}\rfloor)}{p(c-1-\lfloor\frac{c-2}{4}\rfloor) - m(\lfloor\frac{c-2}{4}\rfloor+1)}  \geq \left(\frac{3c-2}{2c-4}\right) \left(\dfrac{m}{p}\right)^{ \lfloor\frac{c-2}{4}\rfloor}.$$
Multiplying by $ \left(\frac{\frac{1}{p}}{\frac{1}{p}}\right) \left(\dfrac{m}{p}\right)^{ \lfloor\frac{c-2}{4}\rfloor}$ both sides of the inequality  we obtain  that
\begin{equation}\label{eqmp}\frac{c-1-\lfloor\frac{c-2}{4}\rfloor}{(c-1-\lfloor\frac{c-2}{4}\rfloor) - \frac{m}{p}(\lfloor\frac{c-2}{4}\rfloor+1)} \geq  \left(\frac{3c-2}{2c-4}\right)\left(\dfrac{m}{p}\right)^{ 2\lfloor\frac{c-2}{4}\rfloor}.\end{equation}

On the one hand, since $\frac{p}{m}\geq \frac{ \lfloor\frac{c-2}{4}\rfloor+2}{c-1- \lfloor\frac{c-2}{4}\rfloor}$ it follows that $\frac{m}{p}\leq \frac{ c-1- \lfloor\frac{c-2}{4}\rfloor}{\lfloor\frac{c-2}{4}\rfloor+2}$ and therefore,

 $$\begin{array}{lcl}
 (c-1-\lfloor\frac{c-2}{4}\rfloor) - \frac{m}{p}(\lfloor\frac{c-2}{4}\rfloor+1) &\geq& (c-1-\lfloor\frac{c-2}{4}\rfloor) - \frac{ (c-1-\lfloor\frac{c-2}{4}\rfloor)(\lfloor\frac{c-2}{4}\rfloor+1)}{\lfloor\frac{c-2}{4}\rfloor+2} \\
 &=&  (c-1- \lfloor\frac{c-2}{4}\rfloor)(1-\frac{\lfloor\frac{c-2}{4}\rfloor+1}{\lfloor\frac{c-2}{4}\rfloor+2} )= \frac{c-1- \lfloor\frac{c-2}{4}\rfloor}{\lfloor\frac{c-2}{4}\rfloor+2}.
 \end{array}$$  
 
Then, 
  $$\frac{c-1-\lfloor\frac{c-2}{4}\rfloor}{(c-1-\lfloor\frac{c-2}{4}\rfloor) - \frac{m}{p}(\lfloor\frac{c-2}{4}\rfloor+1)} \leq \frac{c-1-\lfloor\frac{c-2}{4}\rfloor}{\frac{c-1- \lfloor\frac{c-2}{4}\rfloor}{\lfloor\frac{c-2}{4}\rfloor+2}}=\left\lfloor\frac{c-2}{4}\right\rfloor+2$$ and therefore, using inequality  {(\ref{eqmp})},
  $$\left\lfloor\frac{c-2}{4}\right\rfloor+2\geq  \left(\frac{3c-2}{2c-4}\right) \left(\dfrac{m}{p}\right)^{ 2\lfloor\frac{c-2}{4}\rfloor} \geq \frac{3}{2} \left(\dfrac{m}{p}\right)^{ 2\lfloor\frac{c-2}{4}\rfloor}.$$ 
  
   Since  $c\geq 10$,   it follows that $\frac{m}{p}\leq 1.28$.

On the other hand, since $\frac{m}{p}> 1$ it follows that $$\frac{c-1-\lfloor\frac{c-2}{4}\rfloor}{(c-1-\lfloor\frac{c-2}{4}\rfloor) - \frac{m}{p}(\lfloor\frac{c-2}{4}\rfloor+1)} \geq \left(\frac{3c-2}{2c-4}\right) \frac{m}{p}.$$  
Multiplying both sides of the inequality by $\frac{(c-1-\lfloor\frac{c-2}{4}\rfloor) - \frac{m}{p}(\lfloor\frac{c-2}{4}\rfloor+1)}{c-1-\lfloor\frac{c-2}{4}\rfloor} = 1 - \frac{m}{p}\left(\frac{\lfloor\frac{c-2}{4}\rfloor+1}{c-1-\lfloor\frac{c-2}{4}\rfloor}\right) $ we obtain that
$$1 \geq \left(\frac{3c-2}{2c-4}\right) \frac{m}{p} - \left(\frac{3c-2}{2c-4}\right) \left(\frac{m}{p}\right)^2  \left(\frac{\lfloor\frac{c-2}{4}\rfloor+1}{c-1-\lfloor\frac{c-2}{4}\rfloor}\right)$$  and therefore 
$$\left(\frac{3c-2}{2c-4}\right) \left(\frac{m}{p}\right)^2  \left(\frac{\lfloor\frac{c-2}{4}\rfloor+1}{c-1-\lfloor\frac{c-2}{4}\rfloor}\right) - \left(\frac{3c-2}{2c-4}\right) \frac{m}{p} +1 \geq 0. $$ Since 
$\frac{\lfloor\frac{c-2}{4}\rfloor+1}{c-1-\lfloor\frac{c-2}{4}\rfloor} \leq \frac{c+2}{3c-2}$ we see that 
$$\left(\frac{3c-2}{2c-4}\right) \left(\frac{m}{p}\right)^2  \left(\frac{c+2}{3c-2}\right) - \left(\frac{3c-2}{2c-4}\right) \frac{m}{p} +1 = \left(\frac{c+2}{2c-4}\right) \left(\frac{m}{p}\right)^2  - \left(\frac{3c-2}{2c-4}\right) \frac{m}{p} +1 \geq 0.$$

Thus, using the quadratic formula,  $\frac{m}{p} \geq  \frac{4c-8}{2(c+2)}$ or $\frac{m}{p}\leq \frac{2c+4}{2(c+2)}= 1$. Since $m> p$ it follows that $\frac{m}{p} \geq  \frac{4c-8}{2(c+2)}$  and since $c\geq 10$, $\frac{m}{p}\geq \frac{16}{12} = 1.33$ which is a contradiction because  $\frac{m}{p}\leq 1.28$.
From here, the result follows. \hfill${\ \rule{0.5em}{0.5em}}$

\section{Final remarks}

The problem of finding interesting sufficient conditions in order to guarantee a strong partition of multipartite tournaments {in general} seems to be much more complicated and probably would need different techniques. Let $\mathcal{A}(r,c)$ be the set of $r$-balanced $c$-partite tournaments with no strong partition and $\Omega(r,c)=\min \{\omega(G_{r,c}): G_{r,c}\in \mathcal{A}(r,c)\}$. Notice that better lower bounds of $\Omega(r,c)$ leads, by the Main Theorem, to better sufficient conditions for having strong partitions. For now, our corollaries assume that for every partition of a multipartite tournament  of $\mathcal{A}(r,c)$ there exists exactly one vertex with in-degree or out-degree at most $\left\lfloor\frac{c-2}{4}\right\rfloor$, which we think its not a realistic approximation, and thus one may find a better way to estimate  $\Omega(r,c)$.


\begin{thebibliography}{9}

\bibitem{B-J-G}J. Bang-Jensen, G. Gutin, { Digraphs: Theory, Algorithms and Applications}, Springer, London, 2001.

\bibitem{amo}A.P. Figueroa, J. J.Montellano Ballesteros, M. Olsen, Strong subtournaments and cycles of multipartite tournaments, Discrete Math \textbf{339}  (2016),  2793--2803.

\bibitem{sub}L. Volkmann, Strong subtournaments of multipartite tournaments, Australas J  Combin  \textbf{20}  (1999), 189--196.

\bibitem{Almost-regular} L. Volkmann, S. Winzen, Almost regular c-partite tournaments contain a strong subtournament of order c when $c \geq 5$, Discrete Math  \textbf{308}  (2008), 1710--1721.

\bibitem{X}G.  Xu, S. Li, H. Li, Q. Guo, Strong subtournaments of order image containing a given vertex in regular $c$-partite tournaments with $c\ge 16$, Discrete Math  \textbf{311}  (2011), 2272--2275.


\end{thebibliography}
\end{document}